\documentclass[twoside,a4paper,12pt,centertags]{amsart}
\usepackage{amsmath,amssymb,verbatim,vmargin, amsbsy}
\usepackage[bookmarks=true]{hyperref}   
\usepackage{enumitem}
\usepackage[all]{xy}

\theoremstyle{plain}
\newtheorem{thm}{Theorem}[section]
\newtheorem{lem}[thm]{Lemma}

\newtheorem{prop}[thm]{Proposition}

\theoremstyle{definition}

\newtheorem{rem}[thm]{Remark}

\newtheorem{ex}[thm]{Example}

\mathchardef\semic="303A
\newcommand{\R}{{\mathbf R}}

\newcommand{\Z}{{\mathbf Z}}

\newcommand{\mD}{{\mathcal D}}

\DeclareMathOperator{\re}{Re}
\newcommand{\im}{\text{{\rm Im}}\,}
\newcommand{\sett}[2]{ \{ #1 \, \semic \, #2 \} }
\newcommand{\bigsett}[2]{ \Big\{ #1 \semic #2 \Big\} }

\newcommand{\supp}{\text{{\rm supp}}\,}
\newcommand{\dist}{\text{{\rm dist}}\,}

\newcommand{\barint}{\mbox{$ave \int$}}

\newcommand{\esssup}{\text{{\rm ess sup}}}

\newcommand{\pv}{\text{p.v.}}

\def\barint_#1{\mathchoice
            {\mathop{\vrule width 6pt
height 3 pt depth -2.5pt
                    \kern -8.8pt
\intop}\nolimits_{#1}}%
            {\mathop{\vrule width 5pt height
3 pt depth -2.6pt
                    \kern -6.5pt
\intop}\nolimits_{#1}}%
            {\mathop{\vrule width 5pt height
3 pt depth -2.6pt
                    \kern -6pt
\intop}\nolimits_{#1}}%
            {\mathop{\vrule width 5pt height
3 pt depth -2.6pt
          \kern -6pt \intop}\nolimits_{#1}}}

\usepackage{color}

\definecolor{gr}{rgb}   {0.,   0.8,   0. }
\definecolor{bl}{rgb}   {0.,   0.5,   1. }
\definecolor{mg}{rgb}   {0.7,  0.,    0.7}

\newcommand{\Bk}{\color{black}}
\newcommand{\Rd}{} 

\begin{document}

\title[Carleson-sparse domination]
{Sharp weighted non-tangential maximal estimates via Carleson-sparse domination}
\author[Andreas Ros\'en]{Andreas Ros\'en$\,^1$}
\thanks{$^1\,$Formerly Andreas Axelsson. Supported by the 
Swedish Research Council (Grant 2022-03996).}
\address{Andreas Ros\'en\\Mathematical Sciences, Chalmers University of Technology and University of Gothenburg\\
SE-412 96 G{\"o}teborg, Sweden}
\email{andreas.rosen@chalmers.se}


\begin{abstract}
We prove sharp weighted estimates for the non-tangential maximal function
of singular integrals mapping functions from $\R^n$ to the half-space in $\R^{1+n}$
above $\R^n$. The proof is based on pointwise sparse domination of the adjoint
singular integrals that map functions from the half-space back to the boundary.
It is proved that these map $L_1$ functions
in the half-space to weak $L_1$ functions on the boundary.
From this a non-standard sparse domination of the singular integrals is established, where averages have been replaced by Carleson averages.
\end{abstract}

\maketitle

\section{Introduction}

Let us recall the sparse domination paradigm for estimating singular integral operators, which has been very successful in proving sharp weighted estimates for various singular operators for more than a decade.
Given a singular integral operator $T$ on $\R^n$, with Calder\'on--Zygmund kernel 
$k(x,y)$, the procedure for obtaining estimates is as follows.
\begin{enumerate}[label=(\alph*)]
\item Boundedness of $T$ on $L_2(\R^n)$ is proved. For classical 
convolution singular integral operators the Fourier transform is used, 
and for nonconvolution singular integral operators
$Tb$ theorems are used.
\item Using the Calder\'on--Zygmund decomposition, $L_2$ boundedness 
and estimates of $k(x,y)$, weak $L_1(\R^n)$ boundedness of $T$ is proved.
\item Using weak $L_1(\R^n)$ boundedness of $T$ and its grand 
maximal truncation operator $M_T$, a pointwise domination
$$
  |Tf(x)|\lesssim \sum_{Q\in \mD_f, Q\ni x} \barint_{\!\!\!3Q} |f(y)| dy,
  \qquad x\in \R^n,
$$
by a sum of averages of $f$, over a sparse collection $\mD_f$ of dyadic cubes 
$Q$, is proved.
Sparse roughly means that the cubes in $\mD_f$ are essentially disjoint.
See Section~\ref{sec:carlsparse} for the definition. 
\item
From the sparse domination, one can prove the boundedness of $T$ on any
Banach function space on which, together with its dual space, the maximal function is bounded. For example on weighted $L_p(w)$-spaces, $1<p<\infty$, 
$w\in A_p(\R^n)$. 
\end{enumerate}
See for example Lerner~\cite{Lerner:16} and
Lerner and Nazarov~\cite{LernerNaz:19}. \Rd
We use the analyst's inequality $X\lesssim Y$, which means that
$X\le CY$ for some constant $C<\infty$ independent of relevant variables
but possibly depending on some parameters which should be clear from 
the context. For example, the $C$ in (c) above is independent of $f$ and $x$, but may depend on the sparseness parameter $\eta$.
$X\gtrsim Y$ means $Y\lesssim X$, and $X\eqsim Y$ means $X\lesssim Y$
and $X\gtrsim Y$. \Bk

Somewhat hand in hand with singular integrals $T$ as above, goes the 
theory of singular operator families $\{\Theta_t\}_{t>0}$, where
$\Theta_t$ are integral operators
\begin{equation}  \label{eq:thetat}
   \Theta_tf(x)= \int_{\R^n} k(t,x;y) f(y) dy, \qquad x\in\R^n, t>0.
\end{equation}
Equivalently, the operator family defines a mapping from functions on $\R^n$ to functions on 
the upper half-space $\R^{1+n}_+=\sett{(t,x)}{t>0,x\in\R^n}$,
and a basic $L_2$ estimate is the square function estimate
\begin{equation}    \label{eq:sfe}
   \int_0^\infty \|\Theta_tf\|_2^2 \, \frac{dt}t \lesssim \|f\|_2^2.
\end{equation}
For families of classical convolution operators $\Theta_t$, the 
square function estimate \eqref{eq:sfe} is Calder\'on's reproducing formula and 
the function $(t,x)\mapsto \Theta_tf(x)$ is a continuous wavelet transform 
of $f$. See Daubechies~\cite{Daubechies:92}.
For families of nonconvolution operators $\Theta_t$, there are $Tb$ theorems
for proving \eqref{eq:sfe}. 
See for example Semmes~\cite{Semmes:90}
and Hofmann and Grau De La Herr\'{a}n~\cite{HofmannGrau:17}. 
Sparse domination has been extended to this framework 
to prove sharp weighted estimates of the {\em square function}  
$$
  x\mapsto \left(  \int_0^\infty |\Theta_tf(x)|^2 \frac{dt}{t}\right)^{1/2}.
$$
See Lerner~\cite{Lerner:11} and Bailey, Brocchi and Reguera~\cite{BBReguera:23}. 

In the above works, the kernel of $\Theta_t$ is at least integrable on $\R^n$
and typically satisfies an estimate
\begin{equation}   \label{eq:kernelest}
  |k(t,x;y)|\lesssim t^{-n}\frac 1{(1+ |x-y|/t)^{n+\delta}}, \qquad x,y\in\R^n, t>0, 
\end{equation}
for some $\delta>0$, together with suitable off-diagonal H\"older
estimates.
In the present work, we consider operators $\Theta_t$ with kernels that are only 
locally integrable and typically satisfy \eqref{eq:kernelest} only for $\delta=0$.
The simplest example are the (non-singular) Riesz transforms
$$
  R_t^j f(x)= \int_{\R^n}  \frac{x_j-y_j}{(t^2+|x-y|^2)^{(1+n)/2}} f(y) dy,
   \qquad x\in\R^n, t>0,
$$
where $(t,x)\mapsto R_t^j f(x)$, $j=1,\ldots, n$, are Stein--Weiss harmonic conjugate functions to the Poisson extension 
$$
  P_tf(x)= R_t^0f(x)= \int_{\R^n}  \frac{t}{(t^2+|x-y|^2)^{(1+n)/2}} f(y) dy,\qquad x\in\R^n, t>0,  
$$
of $f$.

If $\delta=0$, the square function estimate \eqref{eq:sfe} is no longer  
suitable as an $L_2$ estimate to feed into a sparse domination scheme. This is because, 
in contrast to the case where the kernels are integrable,
we generally do not have strong convergence $\Theta_t f\to 0$ 
for $\delta=0$ when $t\to 0$,
so the left hand side in \eqref{eq:sfe} is typically infinite.
Instead, a natural object to estimate when $\delta=0$ is the 
non-tangential maximal function of $\Theta_t f(x)$.
We therefore replace \eqref{eq:sfe} with an estimate
\begin{equation}   \label{eq:nte}
   \int_{\R^n}\Big( \sup_{(t,x): |x-z|<\alpha t} |\Theta_t f(x)| \Big)^2 dz
   \lesssim \|f\|_2^2,
\end{equation}
which serves as our starting $L_2$ estimate for step (a) in a sparse domination scheme.
Given an $L_2$ non-tangential estimate \eqref{eq:nte}, we proceed by duality
and consider the map
\begin{equation}   \label{eq:synthesis}
  Sf(y)= \iint_{\R^{1+n}_+} k(t,x;y) f(t,x) dtdx,\qquad y\in\R^n,
\end{equation}
which maps \Rd functions $f(t,x)$ defined \Bk in the upper half-space $\R^{1+n}_+$
to functions $(Sf)(y)$ on its boundary $\R^n$.
By duality, \eqref{eq:nte} corresponds to the Carleson estimate
\begin{equation}   \label{eq:L2Carleson}
   \|Sf\|_2^2\lesssim \int_{\R^n}  \left( \sup_{Q\ni z}\frac 1{|Q|}
   \iint_{\widehat Q} |f(t,x)| dtdx  \right)^2 dz
\end{equation}
for $S$, where the $\sup$ is over all cubes $Q\subset \R^n$ containing $z$,
and $\widehat Q\subset\R^{1+n}_+$ is the Carleson box, the cube with $Q$ as its base. \Rd
The equivalence of \eqref{eq:L2Carleson} and \eqref{eq:nte} follows from
\cite[Thm. 3.2]{HytRos:13}. \Bk
Our main result, Theorem~\ref{thm:sparsedom}, shows that \eqref{eq:L2Carleson} implies the pointwise Carleson-sparse domination
\begin{equation}   \label{eq:csparse}
  |Sf(y)|\lesssim \sum_{Q\in \mD_f, Q\ni y} \frac 1{|Q|}\iint_{\widehat{3Q}} |f(t,x)| dtdx, \qquad\text{a.e. } y\in\R^n,
\end{equation}
of $S$.
Here $3Q$ denotes the cube with the same center as $Q$ but three times the
side length.
Note that the Carleson averages in this sparse domination normalize by
the $\R^n$ measure $|Q|$, and not by the $\R^{1+n}$ measure $|\widehat{3Q}|\eqsim |\widehat Q|$.
This is in contrast to usual sparse domination techniques.

The four sections of this paper simply implement steps (a)-(d) of
the sparse domination scheme described above.
Starting from an $L_2$ non-tangential maximal estimate \eqref{eq:nte}, or equivalently an $L_2$ Carleson estimate \eqref{eq:L2Carleson}, we prove
weighted $L_p$ estimates. For Muckenhoupt weights $w\in A_p(\R^n)$, 
we estimate the $L_p(\R^n,w)$ operator norms by $[w]_{A_p}^{\max(p,q)/p}$, $1/p+1/q=1$. 
See Theorem~\ref{thm:Sweightest} for the Carleson estimate and 
Theorem~\ref{thm:ntsharp} for the dual non-tangential maximal estimate.
The power of $[w]_{A_p}$ is optimal because it is known to be 
optimal for the 
$\R^n$ singular integral $\Theta_0=\lim_{t\to 0} \Theta_t$, \Rd
see Hyt\"onen~\cite{Hyt:19}, \Bk which the non-tangential
maximal function of $(\Theta_t)_{t>0}$ majorizes. 
Note that for $\delta>0$ the non-tangential maximal function of $\Theta_t f$ is \Rd pointwise bounded \Bk by the maximal function of $f$. In this case the
sharper weighted estimate by $[w]_{A_p}^{q/p}$ follows directly
from Buckley~\cite[Thm. 2.5]{Buckley:93}.

The sparse domination and estimates considered in this paper are 
related to those by Hyt\"onen and Ros\'en~\cite{HytRos:23}.
There, $L_p$ non-tangential maximal estimates for 
causal $\R^{1+n}_+$ Calder\'on--Zygmund operators
$$
   T^+ f(t,x)=\pv\iint_{\R^{1+n}_+} K(t,x;s,y) g(s,y) dsdy, \qquad (s,t)\in \R^{1+n}_+
$$
were proved.
Here causal means the kernel condition $K(t,x;s,y)=0$ for $s>t$,
that is, $T^+$ is upward mapping.
Both those results and the results in the present paper are proved by
a sparse domination of the adjoint anti-causal, or downward mapping,
operator. A natural attempt to prove an estimate of
\eqref{eq:synthesis}, for example in $L_p(\R^n)$-norm, 
is to apply the trace estimate from 
\cite[Thm. 1.1]{HytRos:18} to obtain
$$
  \|Sf\|_{L_p(\R^n)}\lesssim \|C(\nabla \widetilde Sf )\|_{L_p(\R^n)},
$$
where $C$ denotes the Carleson functional, which is defined within
the square in
\eqref{eq:L2Carleson} and $\widetilde S f$ denotes an extension of $Sf$
to $\R^{1+n}_+$, defined by an auxiliary weakly singular integral operator
$\widetilde S$ on  $\R^{1+n}_+$.
The idea would then be to apply \cite[Thm. 5.1]{HytRos:23} to the singular 
integral $T= \nabla \widetilde S$ on  $\R^{1+n}_+$.
However this seems to be impossible. Technically, the boundedness of $T$ would
require a Whitney averaging in the Carleson norm.
More seriously, according to \cite[Ex. 2.1]{HytRos:23} it is necessary 
that $T=T^-$
is anti-causal in order for boundedness to be possible in the Carleson norm.
To achieve anti-causality, a natural way is to truncate the kernel
of $\widetilde S$ so that it becomes anti-causal. 
(This is possible without changing the boundary values $Sf$.) 
But at $t=s$, this adds an $\R^n$ singular integral operator to 
$T= \nabla \widetilde S$, which acts on $f(t,x)$ in the $x$ variable for each $t>0$.
According to \cite[Ex. 2.3]{HytRos:23}, such a horizontal mapping singular integral
will in general not be bounded in a Carleson norm.
To summarize, the estimates considered in this paper are 
related to, but not implied by, those in \cite{HytRos:23}.

\section{Setup and $L_2$ estimates}

This section is about step (a) in the sparse domination scheme.
First some notation. 
We fix a system of dyadic cubes $\mD=\bigcup_{j\in\Z}\mD^j$ in $\R^n$,
where $\mD^j$ are the cubes of side length $\ell(Q)= 2^{-j}$, such that
the dyadic cubes in $\mD$ form a connected tree under inclusion.
Given $Q\in\mD$, its parent is the minimal dyadic cube strictly containing
$Q$, its grandparent is the parent of the parent, its children are the 
maximal dyadic cubes stricly contained in $Q$ and its siblings are
the other children of the parent of $Q$. 
For a given cube $Q\subset\R^n$, dyadic or not, 
we denote by $cQ$, $c>0$, the 
cube with the same center as $Q$ but with side length 
$\ell(cQ)=c\ell(Q)$.
The Carleson box above $Q$ is the $\R^{1+n}_+$ cube
$\widehat Q= (0,\ell(Q))\times Q$, and the Whitney region described
by $Q$ is the upper half
$Q^w= (\ell(Q)/2,\ell(Q))\times Q$
of $\widehat Q$.

With $1_E$ and $|E|$ we denote the indicator function and the measure of a set $E$. The maximal function of a function $f(x)$ on $\R^n$ is
$$
  Mf(x)= \sup_{r>0} \,\barint_{\!\!\! |y-x|<r} |f(y)| dy, \qquad x\in\R^n.
$$
We denote the non-tangential maximal functional
of a function $f(t,x)$ on $\R^{1+n}_+$ by
\begin{equation}    \label{defn:Nfknal}
    Nf(z)= \esssup_{(t,x): |x-z|<\alpha t} |f(t,x)|,\qquad z\in\R^n,
\end{equation}
where the aperture $\alpha>0$ of the cones is a fixed constant. 
We denote the Carleson functional of
a function $f(t,x)$ on $\R^{1+n}_+$ by
  $$
    Cf(z)= \sup_{Q\ni z}\frac 1{|Q|}
   \iint_{\widehat Q} |f(t,x)| dtdx,\qquad z\in\R^n,
$$
where the $\sup$ is over all (non-dyadic) cubes $Q\subset \R^n$ containing $z$.
Versions of $M$, $N$ and $C$, dyadic and non-dyadic, will appear
below where they are needed.

Throughout this paper, we fix a kernel function $k:\R^{1+2n}\to\R$
which morally has the size estimates
$$
  |k(t,x;y)| \lesssim \frac 1{|(t,x)-(0,y)|^{n}}.
$$
However, we will not use this estimate, but only the 
following three.
We assume, for some $\delta>0$, off-diagonal H\"older estimates
\begin{equation}   \label{eq:xreg}
  |k(t+s,x+z;y)-k(t,x;y)|  \lesssim \frac{|(s,z)|^\delta}{|(t,x)-(0,y)|^{n+\delta}}
\end{equation}   
for $|(s,z)|\le  |(t,x)-(0,y)|/2$, and 
\begin{equation}  \label{eq:yreg}
   |k(t,x;y+z)-k(t,x;y)|   \lesssim \frac{|z|^{\delta}}{|(t,x)-(0,y)|^{n+\delta}}
\end{equation}
for $|z|\le  |(t,x)-(0,y)|/2$, where $s,t>0$ and $x,y,z\in\R^n$.

Using the kernel $k(t,x;y)$ we define integral operators $\Theta_t$ as in \eqref{eq:thetat} and an operator $S$ as in 
\eqref{eq:synthesis}, and we assume that $L_2$ Carleson estimates
\begin{equation}   \label{eq:L2Cest}
   \| S f\|_{L_2(\R^n)}\lesssim \|Cf\|_{L_2(\R^n)}
\end{equation}
hold, \Rd that is, \eqref{eq:L2Carleson}. \Bk
By duality, \eqref{eq:L2Cest} is equivalent to $L_2$ non-tangential maximal
 estimates
\begin{equation}   \label{eq:L2Nest}
   \| N(S^* f)\|_{L_2(\R^n)}\lesssim \|f\|_{L_2(\R^n)},
\end{equation} \Rd
that is, \eqref{eq:nte}, \Bk
for the adjoint operator
$$
  S^*f(t,x)=\int_{\R^n} k(t,x;y) f(y) dy,\qquad (t,x)\in\R^{1+n}_+,
$$ \Rd
that is, \eqref{eq:thetat}. \Bk
A standard method for verifying the non-tangential estimate \eqref{eq:L2Nest}
is to derive it from a gradient square function estimate
$$
  \int_0^\infty \|t\nabla \Theta_t f\|_2^2\, \frac{dt}{t} \lesssim \|f\|_2^2,
$$
which in turn can be proved by a $Tb$ theorem applied to the kernel 
$t\nabla_{t,x}k(t,x;y)$.
Indeed, $(t,x)\mapsto \Theta_tf(x)$ often
solves an elliptic equation in applications, where $L_2$ estimates between non-tangential
maximal functionals and square functionals are known.
See \cite[Sec. 10.1]{AxAuscher:09} for a large class of elliptic equations.

For the Riesz transforms mentioned in the introduction, there is the following alternative 
algebraic argument, also for Lipschitz domains.

\begin{ex}
To simplify the notation, we only consider dimension $1+n=2$.
The following argument can be performed in higher dimensions, replacing
the complex algebra with the Clifford algebra.
See \cite[Sec. 8.3]{RosenGMA:19} for the higher dimensional Cauchy
integral that would be used for this argument.

Let $t=\phi(x)$ be the graph of a Lipschitz function $\phi:\R\to\R$, 
and consider the family of operators
$$
  \Theta_t f(x)= \frac 1{2\pi i} \int_\R \frac{f(y) \, (1+i\phi'(y)) dy}
  {y+i\phi(y)-x-i(t+\phi(x))}, \qquad x\in\R, t>0,
$$
which represent the Cauchy integral that acts on functions on the
graph.
If $\phi=0$ and $f$ is real-valued, then $\re\Theta_t f$ is the
Poisson integral of $f$ and $\im \Theta_t f$ is the Riesz transform 
of $f$, but for general $\phi$, both real and imaginary parts have only
kernels with the decay $\delta=0$ in \eqref{eq:kernelest}.
 
We now realise that by the Cauchy integral formula, we have the
algebraic identity
$$
   \Theta_t f= \Theta_t(\Theta_0 f),
$$
since $g=\Theta_0f=\lim_{t\to 0} \Theta_t f$ represents the trace
on the graph of the analytic function $\Theta_t f$ above this graph.
The function $g$ belongs to the upper Hardy subspace, and therefore
its Cauchy extension $\Theta_{-t} g=0$, $t>0$, vanishes below the graph.
Thus
$$
  \Theta_t f= (\Theta_t-\Theta_{-t})(\Theta_0 f), \qquad t>0,
$$ 
where one proves that the kernel of $\Theta_t-\Theta_{-t}$ has
Poisson kernel estimates
$\lesssim t/(t^2+(x-y)^2)$.
From this follows the pointwise maximal estimate
$$
  N(S^* f)(z)\lesssim M(\Theta_0 f)(z), \qquad z\in \R. 
$$
Since $M: L_p(\R^n)\to L_p(\R^n)$ is bounded for $1<p\le \infty$
and the singular Cauchy integral 
$\Theta_0: L_p(\R^n)\to L_p(\R^n)$ is bounded for $1<p< \infty$,
\eqref{eq:L2Nest} follows.
However, the sharp weighted estimates in Theorem~\ref{thm:ntsharp}
below do not follow.
\end{ex}

\section{Weak $L_1$ estimates}

This section is about  step (b) in the sparse domination scheme.

\begin{prop}   \label{prop:cz}
   Assume that the operator $S$ has $L_2$ Carleson estimates
   \eqref{eq:L2Cest} and that the kernel $k$ has $x$-regularity
   \eqref{eq:xreg}. Then $S$ has weak $L_1$-estimates
   $$
     |\sett{y\in\R^n}{|Sf(y)|>\lambda}|\lesssim \lambda^{-1} \|f\|_{L_1(\R^{1+n}_+)}, \qquad \lambda>0.
   $$
\end{prop}

For the proof, we need a twin of the Carleson functional, the area functional
$$
    Af(z)= \iint_{|x-z|<\alpha t} |f(t,x)| \,\frac{dtdx}{t^n},\qquad z\in\R^n,
$$
where the aperture $\alpha>0$ of the cones is a fixed constant. 
Recall that for $1\le p<\infty$,
different apertures $\alpha>0$ give equivalent norms $\|Af\|_{L_p(\R^n)}$.
See \cite[Prop. 2.2]{HytRos:18}.

\begin{lem}   \label{lem:goodlam}
  Let $g:\R^{1+n}_+\to \R$ be a function such that $\|Cg\|_{L_\infty(\R^n)}\le\lambda$. Then 
$$
   \int_{\R^n} |Ag(z)|^2 dz\lesssim \lambda \int_{\R^n} |Ag(z)| dz.
$$
\end{lem}

Recall that 
\begin{equation}   \label{eq:ACequiv}
  \| Af\|_{L_p(\R^n)}\eqsim \| Cf\|_{L_p(\R^n)}
\end{equation}
for $1< p<\infty$. 
See \cite[Prop. 2.4]{HytRos:18}.
However, for $p=\infty$ the area functional can be unbounded, even if the Carleson functional is bounded.
We take the opportunity to correct a typo in the counterexample given
in the first display after \cite[Prop. 2.4]{HytRos:18}: It should read 
$f(t,x)= (t+|x|)^{-1}$ in each dimension $n$.
Lemma~\ref{lem:goodlam} means that $Ag$ is close enough to be bounded
for the stated estimate to hold.

\begin{proof}
The following is inspired by the good lambda
inequality from \cite[Thm. 3(a)]{CoifMeyerStein:85}.
Denote by $A^{(\alpha)}g$ and $A^{(\beta)}g$ the area 
functionals with the apertures $\alpha$ and $\beta$ respectively.
Assuming that $\|Cg\|_{L_\infty(\R^n)}\le\lambda$, we claim that
the estimate
$$
  |\sett{z}{A^{(\alpha)}g(z)>2s}|\lesssim (\lambda/s)
   |\sett{z}{A^{(\beta)}g(z)>s}|
$$
holds for all $s>0$, provided that $\alpha< \beta$.
To see this, let $U=\bigcup_j Q_j$ be a Whitney decomposition 
of the open set $U= \sett{z}{A^{(\beta)}g(z)>s}$ into disjoint
cubes $Q_j\subset \R^n$ such that $\ell(Q_j)\eqsim \dist(Q_j,\R^n\setminus  U)$.
It suffices to show 
\begin{equation}   \label{eq:goodlambda}
  |\sett{z\in Q_j}{A^{(\alpha)}g(z)>2s}|\lesssim (\lambda/s) |Q_j|
\end{equation}
for each $Q_j$, since summing over $j$ then gives the desired estimate.
Let $M_j= \sett{z\in Q_j}{A^{(\alpha)}g(z)>2s}$ and choose a point
$z_j\in \R^n\setminus U$ such that $\dist(z_j, Q_j)\eqsim \ell(Q_j)$,
so that $A^{(\beta)}g(z_j)\le s$. 
Now note that 
\begin{multline*}
  \sett{(t,x)}{|x-z|<\alpha t}\\ \subset 
  \sett{(t,x)}{|x-z_j|<\beta t}\cup \sett{(t,x)}{|x-z|<\alpha t, t<c\ell(Q_j)}
\end{multline*}
for all $z\in Q_j$, for some constant $c>1$ only depending on $\alpha,\beta$.
Indeed, if $|x-z|<\alpha t$ and $t\ge c\ell(Q_j)$, then 
$|x-z_j|\le \alpha t+|z-z_j|$ where $|z-z_j|\eqsim \ell(Q_j)\le t/c$.
This gives the estimate
\begin{multline*}
  |M_j|\le \frac 1{2s}\int_{M_j} A^{(\alpha)}g(z) dz \\
  \le \frac 1{2s}\int_{M_j} A^{(\beta)}g(z_j) dz
  + \frac 1{2s}\int_{Q_j} \left( \iint_{|x-z|<\alpha t, t<c\ell(Q_j)} |g(t,x)|t^{-n} dtdx \right) dz \\
  \le |M_j|/2+ c' s^{-1} \iint_{(1+2c\alpha)Q_j\times (0,c\ell(Q_j))} |g(t,x)| dtdx,
\end{multline*}
for some $c'<\infty$.
Since $\|Cg\|_{L_\infty(\R^n)}<\lambda$, this yields 
$|M_j|\lesssim s^{-1}|Q_j|\lambda$  as claimed.

Finally, multiplying \eqref{eq:goodlambda} by $s$ and integrating over 
$s\in(0,\infty)$ gives
\begin{multline*}
   \int_{\R^n} |A^{(\alpha)}g(z)|^2 dz\eqsim
   \int_0^\infty s |\sett{z}{A^{(\alpha)}g(z)>2s}| ds \\
   \lesssim \int_0^\infty \lambda |\sett{z}{A^{(\beta)}g(z)>s}| ds
   = \lambda \int_{\R^n} |A^{(\beta)}g(z)| dz.
\end{multline*}
Since different apertures give equivalent $L_2$ norms and equivalent $L_1$ norms,
this proves the lemma.
\end{proof}

\begin{proof}[Proof of Proposition~\ref{prop:cz}]
(i)
We estimate $f$ using a Calder\'on--Zygmund argument based on 
Carleson averages.
Given $\lambda>0$, let $Q_j\in \mD$ denote the maximal dyadic cubes for which
$$
  \iint_{\widehat{Q_j}} |f(t,x)| dtdx>\lambda |Q_j|.
$$  
Define 
$$
  b_j(t,x)= 
  \begin{cases}
    f(t,x)-\frac 1{|\widehat{Q_j}|}\iint_{\widehat{Q_j}} f(s,y) dsdy, &
    (t,x)\in \widehat{Q_j}, \\
    0, & \text{else},
  \end{cases}
$$
and $g= f-\sum_j b_j$.
Note that in the definition of $b_j$, we do normalize the integral by
$|\widehat{Q_j}|$ and not $|Q_j|$.

(ii)
Estimating first $Sg$, we note that $\|Cg\|_{L_\infty(\R^n)}\lesssim \lambda$.
Indeed, since $g$ is constant equal to the $\widehat{Q_j}$-average of $f$ on each $\widehat{Q_j}$, it follows that 
$|Q|^{-1}\iint_{\widehat Q}|g| dtdx\lesssim |Q_j|^{-1}\iint_{\widehat{Q_j}}|f| dtdx$ if $Q\subset Q_j$.
Note that $\iint_{\widehat{Q_j}}|f| dtdx\le \iint_{\widehat{R}}|f| dtdx\le 
\lambda|R|\eqsim \lambda |Q_j|$, where $R$ is the dyadic parent of $Q_j$.
For $Q$ not contained in any $Q_j$,
we have $|Q|^{-1}\iint_{\widehat Q}|f| dtdx\le \lambda$.

We can now apply \eqref{eq:L2Cest}, \Rd \eqref{eq:ACequiv} \Bk and Lemma~\ref{lem:goodlam} to get
\begin{multline*}
  |\sett{y\in\R^n}{|Sg(y)|>\lambda}|
  \le \lambda^{-2}\int_{\R^n} |Sg(y)|^2 dy \\
  \lesssim \lambda^{-2}\int_{\R^n} |Cg(y)|^2 dy
  \eqsim \lambda^{-2}\int_{\R^n} |Ag(y)|^2 dy \\
  \lesssim \lambda^{-1}\int_{\R^n} |Ag(y)| dy
  \eqsim \lambda^{-1} \|g\|_{L_1(\R^{1+n}_+)}
  \le \lambda^{-1} \|f\|_{L_1(\R^{1+n}_+)}.
\end{multline*}

(iii)
To estimate $Sb_j$, we do the standard estimate 
$$
  |\sett{y}{|S(\sum_j b_j)|>\lambda}|\le \sum_j |3Q_j|
  + \lambda^{-1} \sum_j\int_{\R^n\setminus (3Q_j)} |Sb_j(y)| dy,
$$
where 
$$
  \sum_j |3Q_j|\lesssim \sum_j \lambda^{-1} \iint_{\widehat{Q_j}}|f(t,x)|dtdx\le \lambda^{-1} \|f\|_{L_1(\R^{1+n}_+)}.
$$
For the second term, we use that $\iint_{\widehat{Q_j}} b_j(t,x) dtdx=0$
and \eqref{eq:xreg} to estimate
\begin{multline*}
\int_{\R^n\setminus (3Q_j)} |Sb_j(y)| dy \\=
  \int_{\R^n\setminus (3Q_j)} \left| \iint_{\widehat{Q_j}}
  (k(t,x;y)-k(t_Q,x_Q; y))b_j(t,x) dtdx \right| dy \\
  \lesssim 
   \iint_{\widehat{Q_j}}
   \left( \int_{\R^n\setminus (3Q_j)} \frac{\ell(Q)^\delta}{|(t_Q,x_Q)-(0,y)|^{n+\delta}} dy \right) |b_j(t,x)| dtdx \\
  \lesssim\iint_{\widehat{Q_j}} |b_j(t,x)| dtdx 
   \lesssim\iint_{\widehat{Q_j}} |f(t,x)| dtdx,
\end{multline*}
where $(t_Q,x_Q)$ denotes the center of $\widehat{Q_j}$.
This yields \Rd
$$
  \sum_j\int_{\R^n\setminus (3Q_j)} |Sb_j(y)| dy
  \lesssim \sum_j \iint_{\widehat{Q_j}} |f(t,x)| dtdx \lesssim 
  \|f\|_{L_1(\R^{1+n}_+)},
$$ \Bk
which completes the proof.
\end{proof}

\begin{rem}
In the proof of Proposition~\ref{prop:cz}, we can easily
reduce to the case when
$f$ is constant on each dyadic Whitney region $Q^w$.
Indeed, given $f\in L_1(\R^{1+n}_+)$, we write $f=f_0+f_1$, 
where $\iint_{Q^w} f_0(t,x)dtdx=0$ and $f_1$ is constant on each dyadic
Whitney region $Q^w$. 
For $f_0$, using \eqref{eq:xreg}, we have strong $L_1$ estimates
\begin{multline*}
  \int_{\R^n} |Sf_0(y)| dy= 
  \int_{\R^n} \left| \sum_{Q\in \mD} \iint_{Q^w} 
  (k(t,x;y)-k(t_Q,x_Q; y))f_0(t,x) dtdx\right| dy \\
  \lesssim 
   \sum_{Q\in \mD} \iint_{Q^w}
   \left( \int_{\R^n} \frac{\ell(Q)^\delta}{|(t_Q,x_Q)-(0,y)|^{n+\delta}} dy \right) |f_0(t,x)| dtdx \\
   \lesssim \sum_{Q\in \mD} \iint_{Q^w} |f_0(t,x)| dtdx = \|f_0\|_{L_1(\R^{1+n}_+)},
\end{multline*}
where $(t_Q,x_Q)$ now denotes the center of $Q^w$.
Thus it remains to estimate $Sf_1$.
\end{rem}

For the sparse domination of $S$, we require the maximal operator
$$
  M_S f(y)= \sup_{Q\ni y} \|S(1_{\R^{1+n}_+\setminus\widehat{3Q}}f)\|_{L_\infty(Q)}, \qquad y\in\R^n,
$$
where the sup is taken over all dyadic cubes $Q\subset\R^n$ containing $y$.
This is a version of Lerner's grand maximal truncation operator from \cite{Lerner:16}, which is enough for our purposes.

\begin{prop}   \label{prop:lernermax}
   Assume that the operator $S$ has $L_2$ Carleson estimates
   \eqref{eq:L2Cest} and that the kernel $k$ has the $x$-regularity
   \eqref{eq:xreg} and the $y$-regularity
   \eqref{eq:yreg}. 
   Then $M_S$ has the weak $L_1$ bound
   $$
     |\sett{y\in\R^n}{|M_Sf(y)|>\lambda}|\lesssim \lambda^{-1} \|f\|_{L_1(\R^{1+n}_+)}, \qquad \lambda>0.
   $$
\end{prop}

\begin{proof}
This result follows by tweaking the standard proof of Cotlar's lemma.
Fix a cube $Q\subset \R^n$.
Let
$$
  f_1= 1_{\widehat{3Q}}f \quad\text{and}\quad
  f_2= 1_{\R^{1+n}_+\setminus\widehat{3Q}}f.
$$ 
At a given point $y_1\in Q$, we estimate $Sf_2(y_1)$ by averaging over
a variable point $y\in Q$. Write
$$
  Sf_2(y_1)= (Sf_2(y_1)-Sf_2(y)) + Sf(y)-Sf_1(y)=
   I+II+III.
$$
Fixing $p>1$, we estimate term II as
$$
  \barint_{\!\!\! Q}|Sf(y)|^{1/p} dy\lesssim \inf_Q M(|Sf|^{1/p})
$$
and term III as
$$
   \barint_{\!\!\! Q}|Sf_1(y)|^{1/p} dy\lesssim |Q|^{-1/p}\|f_1\|_{L_1(\R^{1+n}_+)}^{1/p}\lesssim (\inf_Q Cf)^{1/p},
$$
using Kolmogorov's inequality and Proposition~\ref{prop:cz}.
To estimate term I, we use \eqref{eq:yreg} to get
$$
  |Sf_2(y_1)-Sf_2(y)|\lesssim \iint_{\R^{1+n}_+\setminus\widehat{3Q}}
  \frac{\ell(Q)^\delta}{|(t,x)-(0,y_0)|^{n+\delta}} |f(t,x)| dtdx,
$$
for any $y_0\in Q$.
Writing $r=|(t,x)-(0,y_0)|$ and 
$r^{-n-\delta}\eqsim \int_r^\infty s^{-n-1-\delta} ds$, we get
\begin{multline*}
  |Sf_2(y_1)-Sf_2(y)|\lesssim \ell(Q)^\delta\iint_{\R^{1+n}_+\setminus\widehat{3Q}}
  \left(\int_r^\infty \frac{ds}{s^{n+1+\delta}}\right) |f(t,x)| dtdx\\
  = \ell(Q)^\delta\int_{\ell(Q)}^\infty\frac{ds}{s^{n+1+\delta}}
  \left( \iint_{\{|(t,x-y_0)|<s\}\setminus\widehat{3Q}} |f(t,x)| dtdx\right) \\
  \lesssim \ell(Q)^\delta\int_{\ell(Q)}^\infty\frac{ds}{s^{1+\delta}} Cf(y_0)
  \eqsim Cf(y_0).
\end{multline*}
Collecting the estimates, we have prove the pointwise estimate
$$
  M_S f\lesssim Cf+ M(|Sf|^{1/p})^p + Cf.
$$
Here $M$ is bounded on $L_{p,\infty}(\R^n)$, which combined with
Proposition~\ref{prop:cz} gives the estimate for the second term.
A standard Vitali covering argument finally shows that 
$$
     |\sett{y\in\R^n}{|Cf(y)|>\lambda}|\lesssim \lambda^{-1} \|f\|_{L_1(\R^{1+n}_+)}, \qquad \lambda>0,
$$
which completes the proof.
\end{proof}

\section{Carleson-sparse domination}  \label{sec:carlsparse}

This section is about  step (c) in the sparse domination scheme.

We recall that a collection of dyadic cubes $\widetilde\mD\subset\mD$
is called $\eta$-sparse, $\eta>0$, if each $Q\in \widetilde \mD$ contains
a subset $E_Q\subset Q$ such that $|E_Q|\ge \eta |Q|$ and 
$E_R\cap E_Q=\emptyset$ whever $R\ne Q$, $Q,R\in\widetilde\mD$.

Our main result in this paper is the following Carleson-sparse domination.
The proof given below is an adaption of the estimate in \cite{Lerner:16}
to singular integrals mapping from $\R^{1+n}_+$ to $\R^n$.

\begin{thm}   \label{thm:sparsedom}
   Assume that the operator $S$ has $L_2$ Carleson estimates
   \eqref{eq:L2Cest} and that the kernel $k$ has the $x$-regularity
   \eqref{eq:xreg} and the $y$-regularity
   \eqref{eq:yreg}. 
   Fix $0<\eta<1$.
   Then for any $f\in L_1(\R^{1+n}_+)$ with bounded support, there exists 
   an $\eta$-sparse family $\mD_f$ of dyadic cubes such that
$$
  |Sf(y)|\lesssim \sum_{Q\in \mD_f, Q\ni y} \frac 1{|Q|}\iint_{\widehat{3Q}} |f(t,x)| dtdx, \qquad\text{for a.e. } y\in\R^n.
$$
\end{thm}

\begin{proof}
(i) \Rd
Let $c, \alpha>0$ be constants to be chosen below. \Bk
Let $Q\in \mD$ be such that $\supp f\subset \widehat{3Q}$.
Define
\begin{equation}  \label{eq:Edef}
  E= \bigsett{y\in \R^n}{\max\big(|Sf(y)|, M_S f(y)\big) >c |Q|^{-1}\iint_{\widehat{3Q}} |f|dtdx},
\end{equation}
and let $R_j\in\mD$ be the maximal subcubes $R_j\subset Q$ such that
$$
  |R_j\cap E|>\alpha |R_j|.
$$
Using maximality, we can also obtain a converse estimate by noting that
$$
  |R_j\cap E|\le |R_j^p\cap E|\le \alpha |R_j^p|= \alpha 2^n |R_j|,
$$ 
where $R^p_j$ denotes the dyadic parent of $R_j$.
We choose $\alpha= 1/(2^n+1)$, so that each $R_j$ contains 
a substantial part of both $E$ and $\R^n\setminus E$ in the sense
that 
\begin{equation}  \label{eq:evendiv}
    \min(|R_j\cap E|, |R_j\setminus E|)\ge |R_j|/(2^n+1).
\end{equation}

Lebesgue's differentiation theorem shows that
\begin{equation}   \label{eq:containE}
  E\cap Q\subset \bigcup_j R_j,
\end{equation}
modulo a set of measure zero.
From Propositions~\ref{prop:cz} and \ref{prop:lernermax}, we obtain
$$
  |E|\lesssim \left(c |Q|^{-1}\iint_{\widehat{3Q}} |f|dtdx\right)^{-1}
  \iint_{\widehat{3Q}} |f|dtdx = |Q|/c.
$$
Since $\bigcup R_j=\{M(1_E)>\alpha\}$, by
the weak $L_1$ boundedness of $M$ we have
$$
\sum_j|R_j|\lesssim \alpha^{-1} |E|\lesssim |Q|/(\alpha c).
$$
Having fixed $\alpha=1/(2^n+1)$, we choose $c$ large enough so that
\begin{equation}   \label{eq:smallbads}
  \sum_j|R_j|\le (1-\eta)|Q|.
\end{equation}

(ii) \Rd
Recall that $\supp f\subset \widehat{3Q}$, so that $1_{\widehat{3Q}}f=f$. \Bk
Now write 
\begin{multline*}
  1_Q S(1_{\widehat{3Q}}f)=
  1_{Q\setminus\bigcup R_j} S(1_{\widehat{3Q}}f) 
  +\sum_j 1_{R_j} S(1_{\R^{1+n}_+\setminus\widehat{3R_j}}f)
  +\sum_j 1_{R_j} S(1_{\widehat{3R_j}}f) \\
  =I+II +\sum_j 1_{R_j} S(1_{\widehat{3R_j}}f).
\end{multline*}
\begin{itemize}
\item
For term I, by \eqref{eq:containE} we have $y\notin E$ for a.e.
$y\in Q\setminus\bigcup R_j$, and therefore \eqref{eq:Edef} gives
$$
  |S(1_{\widehat{3Q}}f)(y)|\lesssim |Q|^{-1}\iint_{\widehat{3Q}} |f|dtdx.
$$
\item
For subcube $R_j$ in term II, by \eqref{eq:evendiv} there exists 
$y'\in R_j\setminus E$. For all $y\in R_j$, we therefore have
$$
|S(1_{\R^{1+n}_+\setminus\widehat{3R_j}})f(y)|
\lesssim M_Sf(y')\lesssim |Q|^{-1}\iint_{\widehat{3Q}} |f|dtdx,
$$ 
using \eqref{eq:Edef}.
\end{itemize}
To summarize, we have shown that 
\begin{equation}   \label{eq:esttoit}
  1_Q S(1_{\widehat{3Q}}f)\le 
  C_1|Q|^{-1}\iint_{\widehat{3Q}} |f|dtdx+
\sum_j 1_{R_j} S(1_{\widehat{3R_j}}f),
\end{equation}
for some constant $C_1<\infty$,
where the disjoint subcubes satisfy \eqref{eq:smallbads}.

(iii)
We can now iterate \eqref{eq:esttoit} to get the stated sparse estimate
as follows.
Choose $Q_1\in\mD$ such that $\supp f\subset \widehat{Q_1}$.
Set $P_1=Q_1$ and recursively define $P_{j+1}$ to be the dyadic 
parent of $P_j$, $j=1,2\ldots$.
Let $Q_2, Q_3, \ldots$ be an ordering of all the siblings of all the 
cubes $P_j$, $j=1,2,\ldots$.
We obtain a disjoint union $\R^n=\bigcup_{j=1}^\infty Q_j$ modulo
a zero set.

Since $Q_1$ is contained in a sibling of $Q_j$, it follows that 
$3Q_j\supset Q_1$, and therefore 
$\widehat{3Q_j}\supset \supp f$, $j=1,2,\ldots$.
Hence
\begin{equation}   \label{eq:topsparse}
  Sf= \sum_j 1_{Q_j} Sf= \sum_j 1_{Q_j} S(1_{\widehat{3Q_j}}f). 
\end{equation}
We now apply the estimate in step (ii) above to each $Q=Q_j$,
which produces first generations of subcubes $R_k\subset Q_j$.
Then we apply the estimate in step (ii) above to each such first
generation subcube $Q=R_k$,
which produces second generations of subcubes.
Continuing recursively in this way, we define the family of dyadic  
cubes $\mD_f$ as the union of all $Q_j$ along with all generations of subcubes $R_k$.
For $Q\in \mD_f$, we define 
$$
   E_Q= Q\setminus\bigcup_k R_k,
$$
where $R_k$ are all the subcubes of $Q$, constucted from $Q$ as in 
step (i) above. It follows from \eqref{eq:smallbads} that $\mD_f$
is $\eta$-sparse.
Combining 
\eqref{eq:topsparse} and recursively \eqref{eq:esttoit}, the stated 
sparse domination of $Sf$ follows.
This completes the proof.
\end{proof}

\section{Sharp weighted estimates}

This section is about  step (d) in the sparse domination scheme.
We first derive weighted $L_p$ Carleson estimates of $S$ from Theorem~\ref{thm:sparsedom}, 
and then use duality between 
the Carleson and non-tangential maximal functionals to obtain
weighted $L_q$ non-tangential maximal estimates of $S^*$.

We fix $1<p<\infty$ and $1/p+1/q=1$, a let $w(x)>0$, $x\in \R^n$, be
an $A_p$ weight, that is,
$$
  [w]_{A_p}^{1/p}= \sup_Q \left(\barint_{\!\!\! Q} wdx \right)^{1/p}\left( \barint_{\!\!\! Q} \nu dx \right)^{1/q}<\infty,
$$ 
where the sup is over all cubes $Q$, and 
$\nu= w^{-q/p}$ is the dual weight.
It is readily checked that $\nu$ is an $A_q$ weight with
$[\nu]_{A_q}^{1/q}= [w]_{A_p}^{1/p}$.
Write $w(Q)= \int_Q dw$ for the $w$-measure of $Q$, where $dw=wdx$,
and similarly for $\nu$. 

Some additional technicalities arise since we aim to prove weighted
estimates which are sharp with the respect to the dependence on
the $A_p$ characteristic $[w]_{A_p}$ of the weight. In particular we need to avoid using 
the doubling property of the measures quantitatively.
To this end, dyadic operators are preferred.
We also need to handle dilations $3Q$ of dyadic cubes $Q\in\mD$
coming from Theorem~\ref{thm:sparsedom}.

\begin{lem}   \label{lem:dyadmax}
Let $1<p\le\infty$. Then we have maximal function estimates
$$
  \|M^{3D}_w f\|_{L_p(w)}\lesssim \|f\|_{L_p(w)}
$$
for the centered dyadic weighted maximal operator
$$
  M^{3D}_wf(z)= \sup_{Q\in\mD, Q\ni z} \frac 1{w(3Q)}\int_{3Q} |f| dw.
$$
The same estimate holds for the standard dyadic weighted 
maximal operator
$$
  M^{D}_wf(z)= \sup_{Q\in\mD, Q\ni z} \frac 1{w(Q)}\int_{Q} |f| dw.
$$
The implicit constant is independent of $w$, but depends on $p$.
\end{lem}

\begin{proof}
  It suffices to consider $M^{3D}_w$ as the proof for $M^{D}_w$
  is similar but simpler.
  For $p=\infty$, clearly 
  $\|M^{3D}_w f\|_{L_\infty(w)}\le \|f\|_{L_\infty(w)}$.
  By the Marcinkiewicz interpolation theorem, it suffices to show the
  weak $L_1$ estimate
  $$
    w(\sett{z}{M^{3D}f(z)>\lambda})\lesssim \lambda^{-1} \|f\|_{L_1(w)}.
  $$
  To see this, for given $\lambda>0$ we consider the dyadic cubes
   $Q\in\mD$
  such that $\iint_{3Q} |f| dw>\lambda w(3Q)$.
  Denote by $\{Q_j\}$ the maximal such cubes.
  Then $Q_j$ are disjoint and 
  $\sett{z}{M^{3D}f(z)>\lambda}= \bigcup_j Q_j$.
  This gives 
\begin{multline*}
    w(\sett{z}{M^{3D}f(z)>\lambda})= \sum_j w(Q_j)
    \le \sum_j w(3Q_j)\\ \le \lambda^{-1} \sum_j \int_{3Q_j} |f| dw=
    \lambda^{-1} \int_{\R^n}\sum_j 1_{3Q_j} |f| dw
    \le \lambda^{-1}3^n \|f\|_{L_1(w)}.
\end{multline*}
\end{proof}

\begin{thm}  \label{thm:Sweightest}
Let $1<p<\infty$.
   Assume that the operator $S$ has estimates
   \eqref{eq:L2Cest} and that the kernel $k$ has estimates
   \eqref{eq:xreg} and \eqref{eq:yreg}.
Then we have estimates
$$
  \|Sf\|_{L_p(w)}\lesssim [w]_{A_p}^{\max(1,q/p)} \|C^D_\nu (f\nu^{-1})\|_{L_p(\nu)},
$$
where the dyadic weighted Carleson functional is
$$
  C^D_\nu f(z)= \sup_{Q\in \mD, Q\ni z} \frac 1{\nu(Q)}\iint_{\widehat Q} |f(t,x)| dt d\nu. 
$$
Here the implicit constant in the estimate is independent of $w$ and $\nu$.
\end{thm}

\begin{proof}
(i)
We first estimate $Sf$ by the auxiliary Carleson functional
$$
  C^{3D}_\nu f(z)= \sup_{Q\in \mD, Q\ni z} \frac 1{\nu(cQ)}\iint_{\widehat{3Q}} |f(t,x)| dt d\nu, 
$$
where $c\ge 1$ is a fixed constant.
Let $g:\R^n\to\R$ be a dual function for which equality holds in 
H\"older's inequality
$\int_{\R^n} (Sf) gdx\le \|Sf\|_{L_p(w)}\|g\|_{L_q(\nu)}$.
Using Theorem~\ref{thm:sparsedom} and adapting the argument 
\cite[Sec. 5]{Lerner:16}, we estimate
\begin{multline*}
  \int_{\R^n} (Sf) gdx \lesssim \sum_{Q\in\mD_f}
  \left( \frac 1{|Q|}\iint_{\widehat{3Q}} |f|dtdx \right)
  \left( \int_Q |g| dx \right) \\
  =\sum_{Q\in\mD_f} B_Q
  \left( \frac {\nu(E_Q)^{1/p}}{\nu(cQ)}\iint_{\widehat{3Q}} |f|dtdx \right)
  \left( \frac{w(E_Q)^{1/q}}{w(cQ)}\int_Q |g| dx \right) \\
    \le \sup_{Q\in \mD} B_Q
  \left(\sum_{Q\in\mD_f} \Big(\frac {1}{\nu(cQ)}\iint_{\widehat{3Q}} |f|dtdx\Big)^p \nu(E_Q) \right)^{1/p} \\
  \left(\sum_{Q\in\mD_f} \Big(\frac 1{w(cQ)}\int_Q |g| dx\Big)^q  w(E_Q) \right)^{1/q} \\
      \le \big(\sup_{Q\in \mD} B_Q \big)
 \|C^{3D}_\nu(f\nu^{-1})\|_{L_p(\nu)}
  \|M^{D}_w(gw^{-1})\|_{L_q(w)}, 
\end{multline*} \Rd
using that $w(cQ)\ge w(Q)$ in the last inequality, \Bk
where $E_Q\subset Q$ are the disjoint ample subsets of $Q\in \mD_f$ and
\begin{multline*}
  B_Q= \frac{\nu(cQ)w(cQ)}{|Q|\nu(E_Q)^{1/p}w(E_Q)^{1/q}}
  \\= \left( \frac{w(cQ)}{w(E_Q)} \right)^{1/q}
  \left( \frac{\nu(cQ)}{\nu(E_Q)} \right)^{1/p}
  \left(\Big( \frac{w(cQ)}{|Q|}\Big)^{1/p} \Big(  \frac{\nu(cQ)}{|Q|} \Big)^{1/q} \right).
\end{multline*}
To estimate $B_Q$, write $a= w(cQ)/w(E_Q)$, $b=\nu(cQ)/\nu(E_Q)$
and $\gamma= [w]_{A_p}^{1/p}$, so that 
\begin{equation}  \label{eq:ab1}
B_Q\lesssim a^{1/q}b^{1/p}\gamma,
\end{equation}
since $|Q|\eqsim |cQ|$.
Using that $|Q|\lesssim |E_Q|\le w(E_Q)^{1/p}\nu(E_Q)^{1/q}$ by H\"older's inequality, we have
\begin{equation}  \label{eq:ab2}
  a^{1/p}b^{1/q}\lesssim \gamma.
\end{equation}
Combining \eqref{eq:ab1} and \eqref{eq:ab2} and noting $a,b\ge 1$, 
we have 
$B_Q\le (\gamma b^{-1/q})^{p/q} b^{1/p}\gamma\le \gamma^p$ if $p\ge 2$
and 
$B_Q\le a^{1/q}(\gamma a^{-1/p})^{q/p}\gamma\le \gamma^q$ if $p\le 2$.
This proves that $\sup_{Q\in \mD} B_Q\lesssim [w]_{A_p}^{\max(1,q/p)}$
and yields the stated estimate of $\|Sf\|_{L_p(w)}$, since 
$$
\|M^{D}_w(gw^{-1})\|_{L_q(w)}\lesssim \|gw^{-1}\|_{L_q(w)}
= \|g\|_{L_q(\nu)}.
$$

(ii)
Next we prove the stated estimate by $C^D_\nu$. By (i), it suffices to show that
\begin{equation}   \label{eq:3DvsD}
  \|C^{3D}_\nu f\|_{L_p(\nu)}\lesssim \|C^D_\nu f\|_{L_p(\nu)},
\end{equation}
for some $c\ge 1$ \Rd in the definition of $C^{3D}_\nu f$. \Bk
Assume that $C^{3D}_\nu f(z)>\lambda$.
Then there exists $Q\in\mD$ such that $Q\ni z$ and
$$
  \iint_{\widehat{3Q}} |f|dtd\nu>\lambda \nu(cQ).
$$
Let $P_j\in\mD$ denote the at most $2^n$ dyadic cubes of side length
$4\ell(Q)$ such that $P_j\cap (3Q)\ne\emptyset$, where
$P_1\supset Q$ is the grandparent of $Q$.
By the pigeonhole principle,
$$
\iint_{\widehat{P_j}} |f|dtd\nu \ge 2^{-n}\iint_{\widehat{3Q}} |f|dtd\nu
$$
holds for at least one of the dyadic Carleson regions $\widehat {P_j}$,
since $\bigcup_j \widehat{P_j}\supset \widehat{3Q}$.
We get
$$
\frac 1{\nu(P_j)}\iint_{\widehat{P_j}} |f|dtd\nu \ge 2^{-n}\lambda \nu(cQ)/\nu(P_j),
$$
and so $C^D_\nu f\ge 2^{-n}\lambda \nu(cQ)/\nu(P_j)$ on $P_j$.
Since
$z\in P_1\in\mD$ and $3P_1\supset P_j$, we have
$$
  M_\nu^{3D}(C^D_\nu f)(z)\ge 
  \frac 1{\nu(3P_1)}\int_{P_j} 2^{-n}\lambda \nu(cQ)/\nu(P_j) d\nu
  =\lambda 2^{-n} \nu(cQ)/\nu(3P_1).
$$
Choose $c= 15$. Then $cQ\supset 3P_1$ and therefore $M_\nu^{3D}(C^D_\nu f)(z)\ge \lambda 2^{-n}$.
Letting $\lambda\to C^{3D}_\nu f(z)$, we have shown that
$$
  C^{3D}_\nu f\le 2^n M_\nu^{3D}(C^{D}_\nu f).
$$
The estimate \eqref{eq:3DvsD} now follows from 
Lemma~\ref{lem:dyadmax},
which completes the proof.
\end{proof}

To obtain sharp weighted estimates of $N^D(S^*)$
via duality, where $N^D$ is the dyadic non-tangential maximal functional
$$
   N_Df(z)= \sup_{Q\in \mD, Q\ni z} \|f\|_{L_\infty(Q^w)},\qquad z\in \R^n,
$$
we need the following dyadic weighted duality estimates.

\begin{prop}   \label{prop:weightCarle}
Let $1<p<\infty$.
Then
\begin{equation}   \label{eq:nuCarlesonpair}
\left| \iint_{\R^{1+n}_+} fg dt d\nu \right|
\lesssim \|N^Df\|_{L_q(\nu)} \|C^D_\nu g\|_{L_p(\nu)}.
\end{equation}
Moreover, given any $f:\R^{1+n}_+\to\R$ with $\|N^Df\|_{L_q(\nu)}<\infty$,
there exists a non-vanishing $g:\R^{1+n}_+\to\R$ 
such that
\begin{equation}   \label{eq:revNdual}
   \|N^Df\|_{L_q(\nu)}\|C^D_\nu g\|_{L_p(\nu)} \lesssim  \iint_{\R^{1+n}_+} fg dt d\nu.
\end{equation}
The implicit constants in the estimates are independent of the weight $\nu$.
\end{prop}

\begin{proof}
The result to be proved is a weighted extension of 
\cite[Thm. 2.2]{HytRos:13}, and the proof is a straightforward modification
of the unweighted proof.
Since we shall not use \eqref{eq:nuCarlesonpair}, we only give the 
proof of \eqref{eq:revNdual} to convince the reader that the constants there
are independent of $\nu$.

Given $f$ and $\lambda>0$, we let $\mD_\lambda\subset \mD$ 
be the set of maximal cubes
$Q\in\mD$ such that $\|f\|_{L_\infty(Q^w)}>\lambda$.
Then the cubes in $\mD_\lambda$ are disjoint and we have $\sett{z}{N^D f(z)>\lambda}=\bigcup \mD_\lambda$.
We compute
\begin{multline*}
  \|N^D f\|_{L_q(\nu)}^q=\int_0^\infty \lambda^q \,
  \nu(\{N^D f(z)>\lambda\})\frac{d\lambda}{\lambda} \\
  =\int_0^\infty \lambda^q \sum_{Q\in \mD_\lambda} \nu(Q) \frac{d\lambda}{\lambda} 
  =\sum_{Q\in\mD} \nu(Q) \int_{\lambda: \mD_\lambda\ni Q}
   \lambda^{q}\frac{d\lambda}{\lambda} \\
   \le \sum_{Q\in\mD} \|f\|_{L_\infty(Q^w)}\left(\nu(Q) \int_{\lambda: \mD_\lambda\ni Q}
   \lambda^{q-1}\frac{d\lambda}{\lambda}\right) 
   =   \sum_{Q\in\mD} \|f\|_{L_\infty(Q^w)} g_Q,
\end{multline*}
where we have set $g_Q= \nu(Q) \int_{\lambda: \mD_\lambda\ni Q}
   \lambda^{q-1}\frac{d\lambda}{\lambda}$.
For the inequality, we used that
$\lambda<\|f\|_{L_\infty(Q^w)}$ when $Q\in\mD_\lambda$.
Now define $g$ to be a function on $\R^{1+n}_+$,
whose restriction to $Q^w$ satisfies
$$
  \iint_{Q^w} fg dtd\nu\eqsim \|f\|_{L_\infty(Q^w)} \|g\|_{L_1(Q^w,dtd\nu)}=  \|f\|_{L_\infty(Q^w)}g_Q.
$$

To estimate $C_\nu^D g$, for $Q\in\mD$ we compute
\begin{multline*}
  \frac 1{\nu(Q)}\iint_{\widehat Q} |g| dtd\nu=
  \frac 1{\nu(Q)}\sum_{R\subset Q} \nu(R) 
  \int_{\lambda: \mD_\lambda\ni R} \lambda^{q-1}
  \frac{d\lambda}{\lambda} \\
   = \frac 1{\nu(Q)}\int_0^\infty \lambda^{q-1}
 \sum_{R\in\mD_\lambda, R\subset Q} \nu(R)  \frac{d\lambda}{\lambda} \\
  =\frac 1{\nu(Q)}\int_0^\infty \lambda^{q-1} \,
  \nu(\{N^Df>\lambda\}\cap Q) \frac{d\lambda}{\lambda} \\
  = \frac 1{\nu(Q)}\int_Q (N^D f)^{q-1} d\nu
  \le \inf_Q M^D_\nu((N^D f)^{q-1}).
\end{multline*}
Using Lemma~\ref{lem:dyadmax}, this shows that 
$$
    \|C^D_\nu g\|_{L_p(\nu)}\le \|M^D_\nu((N^D f)^{q-1})\|_{L_p(\nu)}
    \lesssim \|(N^D f)^{q-1}\|_{L_p(\nu)}= \|N^D f\|_{L_q(\nu)}^{q-1}.
$$
Therefore 
$$
   \|N^Df\|_{L_q(\nu)}\|C^D_\nu g\|_{L_p(\nu)} \lesssim
   \|N^Df\|_{L_q(\nu)}^q\le \sum_{Q\in\mD} \|f\|_{L_\infty(Q^w)} g_Q
   \eqsim \iint_{\R^{1+n}_+} fg dt d\nu,
$$
which completes the proof.
\end{proof}

With the above, we are now in position to prove sharp weighted non-tangential
maximal estimates for $\{\Theta_t\}_{t>0}$.

\begin{thm}   \label{thm:ntsharp}
Let $1<q<\infty$.
   Assume that the operator $S^*$ has estimates
   \eqref{eq:L2Nest} and that the kernel $k$ has estimates
   \eqref{eq:xreg} and \eqref{eq:yreg}.
Then for \Rd $f\in L_q(\R^n,\nu)$, \Bk we have estimates
$$
  \|N(S^*f)\|_{L_q(\nu)}\lesssim [\nu]_{A_q}^{\max(1,p/q)} \|f\|_{L_q(\nu)}, 
$$
for any fixed aperture $\alpha>0$ used in the definition \eqref{defn:Nfknal} of $N$.
Here the implicit constant in the estimate is independent of $\nu$.
\end{thm}

\begin{proof}
(i)
The corresponding dyadic estimate, that is, with $N$ replaced
by $N^D$, follows immediately from
 Proposition~\ref{prop:weightCarle} and Theorem~\ref{thm:Sweightest} since
\begin{multline*}
  \|N^D(S^*f)\|_{L_q(\nu)} \lesssim 
   \iint_{\R^{1+n}_+} (S^*f)(t,x)g(t,x) dtd\nu \Big/\|C^D_\nu g\|_{L_p(\nu)} \\
  = \int_{\R^{n}} f(x) \,S(g\nu)(x) dx \Big/  \|C^D_\nu g\|_{L_p(\nu)}\\
\le\|f\|_{L_q(\nu)} \|S(g\nu)\|_{L_p(w)} \Big/ \|C^D_\nu g\|_{L_p(\nu)}\\ 
\lesssim \|f\|_{L_q(\nu)}[w]_{A_p}^{\max(1,q/p)}
= [\nu]_{A_q}^{\max(1,p/q)}\|f\|_{L_q(\nu)},
\end{multline*}
where $g$ is a function dual to $S^*f$, provided by Proposition~\ref{prop:weightCarle}.

(ii)
For the estimate of $N(S^*f)(z)$, $z\in \R^n$, let $(t_z, x_z)$ be such that
$|x_z-z|<\alpha t_z$ and $N(S^*f)(z)\eqsim |S^*f(t_z,x_z)|$.
Write 
$$
  S^*f(t_z,x_z)= (S^*f(t_z,x_z)-S^*f(t_z,z))+ S^*f(t_z,z)= I+II.
$$
For term II, we have $|S^*f(t_z,z)|\le N^Df(z)$.
For term I, we have
$$
  |S^*f(t_z,x_z)-S^*f(t_z,z)|\lesssim\int_{\R^n}
  \frac{|x_z-z|^\delta}{|(t_z,z)-(0,y)|^{n+\delta}} |f(y)|dy. 
$$
Using $|x_z-z|^\delta\lesssim t_z^\delta$ and 
writing 
$$
  \frac 1{(1+r)^{n+\delta}}
  \eqsim \int_r^\infty \frac{ds}{(1+s)^{n+\delta+1}},
$$
this gives 
\begin{multline*}
  |S^*f(t_z,x_z)-S^*f(t_z,z)|\lesssim
  \int_{\R^n} \frac{t_z^{-n}}{(1+|y-z|/t_z)^{n+\delta}}|f(y)| dy \\ \eqsim
  \int_0^\infty t_z^{-n}\left( \int_{|y-z|/t_z<s}|f(y)| dy \right) \frac{ds}{(1+s)^{n+\delta+1}}
  \lesssim Mf(z).
\end{multline*}
Therefore we have the pointwise estimate
$$
  N(S^*f)\lesssim  N^Df+ Mf
$$
The estimate in (i) above and the estimate $\|Mf\|_{L_q(\nu)}\lesssim [\nu]_{A_q}^{p/q} \|f\|_{L_q(\nu)}$ from \cite[Thm. 2.5]{Buckley:93} now completes the proof.
\end{proof}

We note that in the proof of Theorem~\ref{thm:ntsharp}, it would have sufficed
to estimate the vertical maximal function
$\|\sup_{t>0}|\Theta_t f|\|_{L_q(\nu)}$ in step (i).

\bibliographystyle{acm}

\end{document}